\documentclass[11pt]{article} 
\usepackage[colorlinks]{hyperref}
\usepackage{url}
\usepackage[toc,page]{appendix}
\usepackage{mathrsfs}

\usepackage{geometry} 
\usepackage{fixltx2e} 

\usepackage{cmap}

\usepackage[T1]{fontenc}
\usepackage[utf8]{inputenc}
\usepackage{verbatim}

\usepackage{amsmath, amsbsy, amsgen, amscd, amsthm, amsfonts, amssymb} 
\usepackage[all]{xy}

\newtheorem{theorem}{Theorem}

\theoremstyle{myRem}

\theoremstyle{myDef}

\let\originalleft\left
\let\originalright\right
\renewcommand{\left}{\mathopen{}\mathclose\bgroup\originalleft}
\renewcommand{\right}{\aftergroup\egroup\originalright}

\usepackage{mathtools}
\mathtoolsset{centercolon}  

\renewcommand{\phi}{\varphi}

\newcommand{\expect}{\mathbb{E}}
\newcommand{\prob}{\mathbb{P}}

\newcommand{\ff}{\mathcal{F}}
\newcommand{\GG}{\mathcal{G}}

\newcommand{\dirichlet}{\text{{\normalfont Dirichlet}}}
\newcommand{\betavar}{\text{\normalfont Beta}}

\newtheorem{cor}{Corollary}
\newtheorem{prop}{Proposition}
\newtheorem{lem}{Lemma}

\newcommand{\ppp}{\mathscr{P}}
\newcommand{\qqq}{\mathscr{Q}}
\newcommand{\ccc}{\mathcal{C}}
\newcommand{\ddd}{\mathcal{D}}

\title{Confidence Sets for the Source of a Diffusion in Regular Trees}

\author{
Justin Khim and Po-Ling Loh\\ \\ 
Department of Statistics\\
The Wharton School\\
Philadelphia, PA 19104 
}

\begin{document}

\maketitle

\begin{abstract}

We study the problem of identifying the source of a diffusion spreading over a regular tree. When the degree of each node is at least three, we show that it is possible to construct confidence sets for the diffusion source with size independent of the number of infected nodes. Our estimators are motivated by analogous results in the literature concerning identification of the root node in preferential attachment and uniform attachment trees. At the core of our proofs is a probabilistic analysis of P\'{o}lya urns corresponding to the number of uninfected neighbors in specific subtrees of the infection tree. We also provide an example illustrating the shortcomings of source estimation techniques in settings where the underlying graph is asymmetric.

\end{abstract}

\section{Introduction}

With the increasing availability of social network data, developing new mathematical tools to analyze the dynamics of information propagation over a network has become an extremely relevant research area. Diffusion models have long been studied in the sciences to model the spread of an epidemic over a large population (e.g., see \cite{KerMcK27, Rap53, Bai75, And82}, and the references cited in Morris~\cite{Mor93}). More recently, the network structure of individuals in the population was introduced in the analysis~\cite{PasVes01, Stro01, New02}.

Given a particular network structure and the identities of the infected individuals, one problem of interest is to determine the original source of the diffusion process~\cite{shah, frieze, zhu2014, bubeck}. In practice, this could indicate the source of information regarding a natural disaster, leaked classified documents, or a pernicious lie. Shah and Zaman~\cite{shah, shah_universal} considered the problem of diffusion over the nodes of a tree, where the time interval between infections of adjacent nodes was modeled by an exponential random variable. They consequently defined a notion of {\it rumor centrality}, which was shown to be equivalent to maximum likelihood estimation in the case of regular trees. Follow-up papers~\cite{fuchs, karamchandani, zhu2014} include generalizations to recursive trees and scenarios where only incomplete information is provided about the subgraph of infected nodes.

Rather than pinpointing a single node as the possible diffusion source, however, a related goal is to construct a confidence set of nodes with a particular probabilistic guarantee of including the root. This set could then be scrutinized more closely by the investigator. Bubeck et al.~\cite{bubeck} studied this problem in the case of random trees generated according to uniform and preferential attachment models~\cite{barabasi1999}. They showed that confidence sets of a fixed size could be constructed for the root node in each of these cases, regardless of the size of the diffusion. Another related line of work, beginning with with Brautbar and Kearns~\cite{brautbar}, examined the capabilities of local algorithms to identify significant nodes in the graph. Borgs et al.~\cite{borgs} and Frieze and Pegden~\cite{frieze} consequently suggested efficient algorithms used to compute the root node of a preferential attachment graph with high probability, assuming that the algorithm recognizes the root upon encountering it.

Our main contribution is to show that in regular trees, confidence sets of fixed size are also sufficient to identify the source of a diffusion. Our work borrows and extends techniques of Bubeck et al.~\cite{bubeck}, leveraging the fact that the analysis of trees and diffusions, which may be viewed as growing trees subject to a constrained structure, are fundamentally quite similar. We also provide a cautionary example showing that regularity (i.e., symmetry) is crucial to the success of the diffusion source estimators proposed in these papers.

The remainder of the paper is organized as follows: In Section~\ref{SecBackground}, we establish notation and describe the estimators to be used in the paper. Section~\ref{SecMain} presents our main results on source estimators that lead to confidence sets. In Section~\ref{SecAsymmetry}, we provide an example illustrating the shortcomings of the source estimators when the underlying graph is not regular. Section~\ref{SecProofs} contains proofs of our main results, and Section~\ref{SecDiscussion} concludes the paper with several interesting directions for future work.


\section{Background and Estimators}
\label{SecBackground}

Let \(\{T_{n}\}_{n = 1}^{\infty}\) denote the sequence of subtrees obtained from a diffusion process on a regular tree \(\GG\), where the number of neighbors of each node is denoted by $d$. The vertices of \(T_{n}\) are labeled $\{1, 2, \dots, n\}$, according to their order of arrival, and vertex \(n + 1\) is chosen uniformly at random from the neighbors of \(V(T_{n})\) in \(V(\GG) \setminus V(T_{n})\). As noted by Shah and Zaman~\cite{shah}, this model of diffusion may also be viewed as a contagion process over the tree, where the amount of time for the disease to propagate across each edge is exponentially distributed with a certain fixed parameter.

We write \((T, u)\) to denote the tree \(T\) rooted at vertex \(u\), and we write \((T, u)_{v \downarrow}\) to denote the subtree rooted at \(v\) of the rooted tree \((T, u)\). We also denote unlabeled trees using \(\circ\); i.e., \(T^{\circ}\) is a tree with the same topology as \(T\), but with unlabeled vertices. The degree of vertex \(v\) is denoted by \(d_{\GG}(v)\). We will often abuse notation and write \(|T|\) to refer to \(|V(T)|\), the number of vertices in \(T\).

One special property of regular trees is that, given a source node \(1\), 
each realization of \(T_{n}\) has the same probability.
Thus, the likelihood function is proportional to the number of ways in which the diffusion can occur.
This is exactly the quantity
\begin{equation}
R_{T_{n}}(u)
= n! \prod_{v \in T_{n}} \frac{1}{|(T_{n}, u)_{v \downarrow}|}.
\label{eqn: rumor centrality}
\end{equation}
The vertex \(v\) which maximizes \(R_{T_{n}}\) is called the {\it rumor center}
of \(T_{n}\), and it is the maximum likelihood estimate of the source of the diffusion. Conveniently, the rumor center may be calculated in $O(n)$ time via a message-passing algorithm, as detailed in Shah and Zaman~\cite{shah}. We also define the function
\begin{equation}
\phi_{T}(u) 
= \prod_{v \in V(T) \setminus \{u\}} |(T, u)_{v \downarrow}|.
\label{eqn: subtree prod}
\end{equation}
Clearly, minimizing \(\phi_{T_n}\) is equivalent to maximum likelihood estimation. Notably, the estimator~\eqref{eqn: subtree prod} was also studied in Bubeck et al.~\cite{bubeck} and shown to generate a confidence set of constant size in the uniform attachment model, even though it does not represent a maximum likelihood estimator in that case. As shown in Shah and Zaman~\cite{shah}, maximizing the function~\eqref{eqn: rumor centrality} is also equivalent to maximizing the \emph{distance} or \emph{closeness centrality} of a node in the graph, a notion that may be extended beyond trees~\cite{Jac08}. We will also refer to $\varphi_T$ as the subtree product estimator.

We will also analyze the following alternative estimator:
\begin{equation}
\psi_{T}(u) = \max_{v \in V(T) \setminus \{u\}} |(T, u)_{v \downarrow}|,
\label{eqn: max estimator}
\end{equation}
which may be viewed as a relaxation of $\varphi_T(u)$. Note that $\psi_T$ may also be computed in $O(n)$ time. However, the proofs for \(\psi_{T}\) tend to be much easier, as demonstrated in Bubeck et al.~\cite{bubeck}. We will refer to $\psi_T$ as the maximum subtree estimator. The minimizer of $\psi_T$ is also known in the graph theory literature as the \emph{centroid} of the tree~\cite{Ore65}.

Our overall goal is to understand how the size of the required confidence set evolves as a function of the error tolerance and the number of infected nodes. In the results that follow, we will use $H_{\varphi, K}$ and $H_{\psi, K}$ to denote the sets obtained by selecting the $K$ nodes with smallest values of $\varphi_T$ and $\psi_T$, respectively. Thus, we wish to determine the size of $K$ required to ensure that node 1 is contained in $H$ with sufficiently high probability.

\section{Main Results}
\label{SecMain}

In this section, we present theoretical guarantees for the two estimators discussed in the previous section.
The proofs of the theorems may be found in Section~\ref{SecProofs}.

The results in this section assume \(d \geq 3\). When \(d = 2\), the underlying network is a line graph, and the source estimators \(\psi_T\) and \(\phi_T\) are both minimized by selecting the set of nodes at the center of the infection tree. However, it can easily be seen by considering a simple random walk that constructing a confidence set of a fixed size is impossible in this case.


\subsection{Maximum Subtree Analysis}
\label{section: max subtree}

Let \(H_{\psi, K}\) be the function which takes a tree \(T\) and returns the set of \(K\) vertices in the tree with the smallest \(\psi_{T}\) values.
This is the maximum subtree estimator, 
and it yields a confidence set for the source, as stated in the following theorem:

\begin{theorem}
Let \(\{T_{n}\}\) be a diffusion on a \(d\)-regular tree \(\GG\) with \(d \geq 3\).
Then for any \(\eta \in (0, 1)\) and \(K > 3\), we have
\[
\limsup_{n \to \infty} \prob\left\{1 \not \in H_{\psi, K}(T_{n}^{\circ})\right\}
\leq C_{1} \eta^{1 + \frac{1}{d - 2}} + C_{2} K^{2 + \frac{1}{d - 2}} (1 - \eta)^{K - 1 + \frac{1}{d - 2}},
\]
for some constants \(C_{1}\) and  \(C_{2}\), depending on $d$ but not $K$.
\label{thm: max subtree regular tree}
\end{theorem}

\noindent The proof of Theorem~\ref{thm: max subtree regular tree} is contained in Section~\ref{SecThmSubtree}.

As \(d \to \infty\), the bound obtained in Theorem \ref{thm: max subtree regular tree} converges to the analogous bound for uniform attachment attachment trees derived in Bubeck et al.~\cite{bubeck}.
Intuitively, this is due to the fact that as far as the estimator \(\psi\) is concerned, regular trees of high degree behave like diffusions on uniform attachment trees.

The result of Theorem~\ref{thm: max subtree regular tree}, along with some algebra, implies a sufficient condition on the size of a $1-\epsilon$ confidence set. The proof of the following corollary is provided in Appendix~\ref{AppCorConfidence}.

\begin{cor}
\label{CorConfidence}
For sufficiently small $\epsilon > 0$, we have
\begin{equation*}
\limsup_{n \rightarrow \infty} \prob \left\{1 \notin H_{\psi, K} (T^\circ_n)\right\} \le \epsilon,
\end{equation*}
provided $K \ge \frac{C}{\epsilon}$, where $C > 0$ is a constant depending on $d$ but not $\epsilon$.
\end{cor}

In particular, the required size $K$ of a $1-\epsilon$ confidence set is bounded by a function of $\epsilon$ that does not depend on $n$, the number of infected nodes.



\subsection{Subtree Product Analysis}
\label{section: subtree product}

For the function \(\phi_{T}\),
we first find a bound on the probability that 
the rumor center is far from the source of the diffusion.
This could be applied directly in cases where we wish to find a starting point to query nodes
in search of the source. Furthermore, the fact that we may obtain a confidence set of bounded size for the source follows as an immediate corollary. The proof of the following theorem is provided in Section~\ref{SecThmRumor}.

\begin{theorem}
Let \(\{T_{n}\}\) be a diffusion on a \(d\)-regular tree \(\GG\) with \(d \geq 3\),
and let \(v\) be a vertex in \(\GG\).
Define \(\ell(v)\) to be the length of the path between \(1\) and \(v\).
Then for $\ell(v) > 1$, we have
\[
\limsup_{n \to \infty } 
\prob\left\{\phi_{T_{n}}(v) \leq \phi_{T_{n}}(1)\right\}   
\leq 
7 \exp\left(-\frac{\ell(v)}{2} \log\left(\min\left\{\frac{\ell(v)(d-2)}{4e \log(\ell(v))}, \frac{\ell(v)}{2}\right\}\right) \right).
\]
\label{thm: inv rumor centrality distance}
\end{theorem}

From here, we can obtain a confidence set for the source by selecting all nodes
sufficiently close to the rumor center. This is captured in the following corollary, proved in Appendix~\ref{AppCorPhiEst}:

\begin{cor}
\label{CorPhiEst}
Let \(\{T_{n}\}\) be a diffusion on a \(d\)-regular tree \(\GG\) with \(d \geq 3\).
Define \(H_{\phi, L}(T)\)
to be the set of all vertices \(u\) in \(T\) such that
the distance between \(u\) and the rumor center of \(T\) is less than or equal to \(L\).
Then for $L \ge 2$, we have
\begin{equation}
\label{EqnConfWidth}
\limsup_{n \to \infty}
\prob\left\{1 \not \in H_{\phi, L}(T_{n}^{\circ})\right\}
\leq 
7 \exp\left(-\frac{L}{2} \log\left(\min\left\{\frac{L(d - 2)}{4 e d^{2} \log(L)}, \frac{L}{2d^2}\right\}\right) \right).
\end{equation}
In particular, for \(\epsilon > 0\) sufficiently small, we have
\[
\limsup_{n \to \infty}
\prob\left\{1 \not \in H_{\phi, L}(T_{n}^{\circ})\right\}
\leq 
\epsilon,
\]
provided  
\(
L \geq \frac{a\log\left(\frac{7}{\epsilon}\right)}{\log\left(b \log\left(\frac{7}{\epsilon}\right) \right)},
\)
where $a$ and $b$ are constants depending only on \(d\).
\label{cor: rc estimating set}
\end{cor}

Note that Corollary~\ref{CorPhiEst} implies the existence of a $1- \epsilon$ confidence set that is sublinear in $\frac{1}{\epsilon}$, since \mbox{\(K = \frac{d(d-1)^L - 2}{d-2}\)} for the set consisting of all nodes at a distance of at most $L$ from the rumor center, and we may check that $a > 1$. This is a strictly better result than the guarantee provided by Corollary~\ref{CorConfidence}, which makes sense, since $\varphi$ corresponds to a maximum likelihood calculation.


\section{An Example of Asymmetry}
\label{SecAsymmetry}

In this section, we provide an example of a graph \(\GG^{*}\) where the estimators \(\psi\) and \(\phi\) fail to select the source of the diffusion.
Let \(d\) and \(D\) be constants such that \(3 \leq d < D\), and
consider a \(d\)-regular tree \(\GG_{d}\) and a \(D\)-regular tree \(\GG_{D}\).
Pick a vertex \(v_{*}\) from \(\GG_{d}\) and vertex \(v^{*}\) from \(\GG_{D}\), and let \(\GG^{*}\) denote the graph obtained by connecting \(v_{*}\) and \(v^{*}\) with an edge.
We will henceforth think of \(\GG_{d}\) and \(\GG_{D}\) as subgraphs of \(\GG^{*}\).
Furthermore, we will say that a diffusion \(\{T_{n}\}\) on \(\GG^{*}\) has reached a vertex \(u\) if \(u \in V(T_{n})\) for some $n$.

Note that the tree \(\GG^{*}\) lacks the symmetry of the regular graphs we have considered before.
The first consequence is that given a diffusion in \(\GG^{*}\), the different sample paths do not necessarily have the same probability.
Therefore, counting sample paths will not lead to a maximum likelihood estimator.
The more important fact is that \(\psi\) and \(\phi\) may supply estimates for the source node that are arbitrarily far away from the actual source as \(n\) goes to infinity.
This is in stark contrast to the uniform and preferential attachment trees of Bubeck et al.~\cite{bubeck} and the geometric trees of Shah and Zaman~\cite{shah}, for which neither \(\psi\) nor \(\phi\) is a maximum likelihood estimator, but both estimators still perform well. The proof of the following proposition is provided in Appendix~\ref{AppPropBad}.

\begin{prop}
Let \(\{T_{n}\}\) be a diffusion on \(\GG^{*}\), 
and suppose the source vertex \(1\) is in \(\GG_{d}\).
Then \(\phi\) and \(\psi\) satisfy
\[
\limsup_{n \to \infty}
\prob\left\{\phi_{T_{n}}(v^{*}) < \phi_{T_{n}}(1) \right\}
=
\limsup_{n \to \infty}
\prob\left\{\psi_{T_{n}}(v^{*}) < \psi_{T_{n}}(1)\right\}
= 1.
\]
\label{prop: phi psi bad}
\end{prop}


While the important observation is that the graph \(\GG^{*}\) provides an example of where asymmetry causes our subtree-based estimators to fail,
we can actually still produce a confidence set for the source.
We make the obvious modification of breaking \(\GG^{*}\) into \(\GG_{d}\) and \(\GG_{D}\) 
and then evaluating \(\phi\) or \(\psi\) on each part of the graph separately.
We then output the vertices within a distance \(L + 1\) of the rumor center on our two subtrees as our confidence set for the source. The proof of the following proposition is provided in Appendix~\ref{AppPropIrreg}.

\begin{prop}
Let \(\{T_{n}\}\) be a diffusion on \(\GG^{*}\).
Let \(T_{n, d}\) and \(T_{n, D}\) denote the subtrees of the diffusion on \(\GG_{d}\) and \(\GG_{D}\), respectively.
Define \(H_{\phi, L}(T_{n}^{\circ})\) to be the union \(H_{\phi, L}(T_{n, d}^{\circ}) \cup H_{\phi, L} (T_{n, D}^{\circ})\). 
Then we have
\[
\limsup_{n \to \infty} \prob\left\{1 \not \in H_{\phi, L}(T_{n}^{\circ})\right\}
\leq 
7 \exp\left(-\frac{L}{2} \log \left(\min\left\{\frac{L(D - 2)}{4 e D^{2} \log(L)}, \frac{L}{2D^2}\right\}\right)\right).
\]
\label{prop: irregular confidence}
\end{prop}

\noindent Note that we could prove a similar theorem quite easily for the estimator \(\psi\).

However, the result of Proposition~\ref{prop: irregular confidence} is still somewhat unsatisfying, since we require outputting more vertices in order to obtain the same level of confidence, and it does not readily generalize to other trees or graphs. An important future avenue of research is to devise a better estimating procedure for the graph \(\GG^{*}\) as another step toward dealing with more general graphs.


\section{Proofs of Theorems}
\label{SecProofs}

In this section, we provide proof outlines for the theorems stated in Section~\ref{SecMain}.

\subsection{Proof of Theorem~\ref{thm: max subtree regular tree}}
\label{SecThmSubtree}

Our proof mostly follows the analysis of Theorem 3 in Bubeck et al.~\cite{bubeck}.

Recall that we label vertices by the order of their appearance in \(T_{n}\). Let \(T_{i, K}^{n}\) denote the tree that contains vertex \(i\) in the forest obtained from \(T_{n}\) by removing all edges between the vertices \(\{1, \ldots, K\}\). We  have the following lemma, proved in Appendix~\ref{AppLemDirichlet}:

\begin{lem}
\label{LemDirichlet}
The vector of subtree proportions converges in distribution:
\begin{equation*}
\left(\frac{|T^n_{1,K}|}{n}, \dots, \frac{|T^n_{K,K}|}{n}\right) \stackrel{d}{\longrightarrow} \dirichlet\left(\frac{d_{\GG}(1) - d_{T_{K}}(1)}{d - 2},
\dots, \frac{d_{\GG}(K) - d_{T_{K}}(K)}{d - 2}\right).
\end{equation*}
\end{lem}

Let \(\eta \in (0, 1)\), and note that
\begin{align}
\prob\left\{1 \not \in H_{\psi, K}(T_n^\circ)\right\} 
&\leq \prob\left\{\exists i > K : \psi_{T_{n}}(i) \leq \psi_{T_{n}}(1)\right\} \notag \\ 
& \leq \prob\left\{(1 - \eta)n \leq \psi_{T_{n}}(1)\right\} 
+ \prob\left\{\exists i > K: \psi_{T_{n}}(i) \leq (1 - \eta)n\right\}.
\label{eqn: thm split}
\end{align}
We now bound each of these terms. For the first term, note that \(\psi_{T_{n}}(1) \leq \max(|T_{1, 2}^{n}|, |T_{2, 2}^{n}|).\) By Lemma~\ref{LemDirichlet}, we know that \(|T_{1, 2}^{n}| / n\) and \(|T_{2, 2}^{n}| / n\) are identically distributed and converge to a \(\betavar \left(\frac{d - 1}{d - 2}, \frac{d - 1}{d - 2}\right)\) random variable.
Thus, we obtain
\begin{align*}
  &\begin{aligned}
\limsup_{n \to \infty} \prob\left\{(1 - \eta)n \leq \psi_{T_{n}}(1)\right\}  
&\leq 2 \limsup_{n \to \infty} \prob\left\{(1 - \eta)n \leq |T_{1, 2}^{n}|\right\} \\ 
&= 
2  \frac{1}{\beta\left(\frac{d - 1}{d - 2}, \frac{d - 1}{d - 2}\right)}  
\int_{1 - \eta}^{1} 
x^{\frac{1}{d - 2}} (1 - x)^{\frac{1}{d - 2}} \; dx\\ 
&\leq \frac{2}{\beta\left(\frac{d - 1}{d - 2}, \frac{d - 1}{d - 2}\right)}
\int_{0}^{\eta}  x^{\frac{1}{d - 2}} \; dx \\ 
&= C_{1} \eta^{1 + \frac{1}{d - 2}},
  \end{aligned}
\end{align*}
where 
\(C_{1} = \frac{2(d - 2)}{(d - 1)\beta\left(\frac{d - 1}{d - 2}, \frac{d - 1}{d - 2}\right)}\).

To bound the second term, note that for  \(i > K\), we have
\(\psi_{T_{n}}(i) \geq \min_{1 \leq k \leq K} \sum_{j = 1, j \neq k}^{K} |T_{j, K}^{n}|\).
For ease of exposition, let \(W_{k, K}^{n} = \sum_{j = 1, j \neq k}^{K} |T_{j, K}^{n}|\).
We then have
\begin{align*}
  &\begin{aligned}
\limsup_{n \to \infty} 
\prob\left\{\exists i > K: \psi_{T_{n}}(i) \leq (1 - \eta)n \right\} 
& \leq \lim_{n \to \infty} 
\prob\left\{\exists k: W_{k, K}^{n} \leq (1 - \eta)n\right\} \\ 
&\leq \lim_{n \to \infty} \sum_{k = 1}^{K} 
\prob\left\{W_{k, K}^{n} \leq (1 - \eta)n \right\}.
  \end{aligned}
\end{align*}

The next lemma follows by integrating the appropriate probability density functions obtained from the convergence result in Lemma~\ref{LemDirichlet}. The proof is provided in Appendix~\ref{AppLemWtail}.

\begin{lem}
\label{LemWtail}
For each $1 \le k \le K$, we have the inequality
\begin{equation*}
\lim_{n \rightarrow \infty} \prob\left\{W^n_{k, K} \le (1-\eta) n\right\} \le C_2 K^{1 + \frac{1}{d - 2}} (1 - \eta)^{K - 1 + \frac{1}{d - 2}}.
\end{equation*}
\end{lem}

Substituting the bounds into inequality~\eqref{eqn: thm split}, we arrive at the desired result.


\subsection{Proof of Theorem~\ref{thm: inv rumor centrality distance}}
\label{SecThmRumor}

Our proof is largely inspired by the proof of Theorem 5 in Bubeck et al.~\cite{bubeck}.

Let \(v\) be a vertex in \(V(\GG)\).
We also write
\(v = (j_{1}, \ldots, j_{\ell(v)})\), where \((j_{1}, \ldots, j_{i})\) is the \(j_{i}^\text{th}\) child of 
\((j_{1}, \ldots, j_{i - 1})\) when viewing \(1\) as the root of the tree.
Note that \(\phi_{T_{n}}(v) \leq \phi_{T_{n}}(1)\)
if and only if
\begin{equation}
\prod_{i = 1}^{\ell(v)} (n - |(T_{n}, 1)_{(j_{1}, \ldots, j_{i}) \downarrow}|)
=
\prod_{i = 0}^{\ell(v) - 1} |(T_{n}, v)_{(j_{1}, \ldots, j_{i}) \downarrow}|
\leq 
\prod_{i = 1}^{\ell (v)} |(T_{n}, 1)_{(j_{1}, \ldots, j_{i}) \downarrow}|.
\label{eqn: phiv condition}
\end{equation}
Dividing each term in the product by \(n\), this is equivalent to
\[
\prod_{i = 1}^{\ell(v)} \left(1 - \frac{1}{n}|(T_{n}, 1)_{(j_{1}, \ldots, j_{i}) \downarrow}|\right)
\leq 
\prod_{i = 1}^{\ell (v)} \frac{1}{n} |(T_{n}, 1)_{(j_{1}, \ldots, j_{i}) \downarrow}|.
\]

Since we are interested in the limiting probability \(\limsup_{n \to \infty}
\prob\left\{\phi_{T_{n}}(v) \leq \phi_{T_{n}}(1)\right\}\),
we would like to understand the random variables
\(|(T_{n}, 1)_{(j_{1}, \ldots, j_{i}) \downarrow}| / n\). The following lemma provides the required convergence result, and is proved via a P\'{o}lya urn analysis. The proof is contained in Appendix~\ref{AppLemBetaConv}.

\begin{lem}
\label{LemBetaConv}
We have the convergence in distribution
\begin{multline*}
\left(\frac{|(T_n, 1)_{(j_1) \downarrow}|}{n}, \frac{|(T_n, 1)_{(j_1, j_2) \downarrow}|}{n}, \cdots, \frac{|(T_n, 1)_{(j_1, \dots, j_{\ell(v)}) \downarrow}|}{n}\right) \\
\stackrel{d}{\longrightarrow} \left(B_{(j_1)}, \; B_{(j_1)} B_{(j_1, j_2)}, \; \dots, \; \prod_{i=1}^{\ell(v)} B_{(j_1, \dots, j_{\ell(v)})}\right),
\end{multline*}
where the \(B_{(j_{1}, \ldots, j_{i})}\)'s  are independent \(\betavar\left(\frac{1}{d - 2}, \frac{d - 1}{d - 2}\right)\) random variables for \(i = 1\), and \(\betavar\left(\frac{1}{d - 2}, 1\right)\) random variables for \(i > 1\).
\end{lem}

By Lemma~\ref{LemBetaConv}, it therefore follows that
\begin{equation}
\limsup_{n \to \infty}
\prob\left\{\phi_{T_{n}}(v) \leq \phi_{T_{n}}(1)\right\}
= 
\prob\left\{
\prod_{i = 1}^{\ell(v)} \left(1 - \prod_{k = 1}^{i} B_{(j_{1}, \ldots, j_{k})}\right)
\leq 
\prod_{i = 1}^{\ell(v)} \prod_{k = 1}^{i} B_{(j_{1}, \ldots, j_{k})} 
\right\}.
\label{eqn: mixed betas}
\end{equation}
Now, it would simplify the analysis if all of the 
\(B_{(j_{1}, \ldots, j_{i})}\)
were \(\betavar\left(\frac{1}{d - 2}, 1\right)\) random variables.
By the stochastic domination result proved in Lemma~\ref{lemma: beta dominance} of Appendix~\ref{AppAuxiliary}, we may substitute the 
\(B_{(j_{1}, \ldots, j_{k})}\) for 
\(B_{(j_{1}, \ldots, j_{k})}'\), where the latter set consists only of 
\(\betavar\left(\frac{1}{d - 2}, 1\right)\) random variables.
For a value of $t$ to be chosen later, we then have
\begin{align}
\limsup_{n \to \infty} \prob\left\{\phi_{T_{n}}(v) \leq \phi_{T_{n}}(1)\right\}
& \leq \prob\left\{\prod_{i = 1}^{\ell(v)} \left(1 - \prod_{k = 1}^{i} B_{(j_{1}, \ldots, j_{k})}'\right)
\leq \prod_{i = 1}^{\ell(v)} \prod_{k = 1}^{i} B_{(j_{1}, \ldots, j_{k})}' \right\} \notag \\
& \leq \prob\left\{\prod_{i = 1}^{\ell(v)} \left(1 - \prod_{k = 1}^{i} B_{(j_{1}, \ldots, j_{k})}'\right) 
\leq \exp(-t) \right\} \notag \\
& \qquad + \prob\left\{\exp(-t) \leq \prod_{i = 1}^{\ell(v)} \prod_{k = 1}^{i} B_{(j_{1}, \ldots, j_{k})}'\right\}.
\label{eqn: split beta prob}
\end{align}
We now bound each term on the right-hand side of inequality~\eqref{eqn: split beta prob} separately. We have the following lemmas, proved in Appendices~\ref{AppLemFirst} and~\ref{AppLemSecond}:

\begin{lem}
\label{LemFirst}
For any $t \in \mathbb R$, we have
\begin{equation}
\label{eqn: lemma 3 bound}
\prob\left\{\prod_{i = 1}^{\ell(v)} \left(1 - \prod_{k = 1}^{i} B_{(j_{1}, \ldots, j_{k})}'\right) 
\leq \exp(-t) \right\} \le 6 \cdot 2^{\frac{\ell(v)}{4}} \exp\left(-\frac{t}{4}\right).
\end{equation}
\end{lem}

\begin{lem}
\label{LemSecond}
For any $t > 0$, we have
\begin{equation}
\prob\left\{\exp(-t)
\leq 
\prod_{i = 1}^{\ell(v)} \prod_{k = 1}^{i} B_{(j_{1}, \ldots, j_{k})}'\right\}
\leq 
\exp\left(
\frac{\ell(v)}{2} - \frac{t}{\ell(v)(d - 2)} - \frac{\ell(v)}{2} \log \left(\frac{\ell(v)^{2}(d-2)}{2t}\right) 
\right).
\label{eqn: lemma 4 bound}
\end{equation}
\end{lem}

Plugging inequalities~\eqref{eqn: lemma 3 bound} and~\eqref{eqn: lemma 4 bound} into inequality~\eqref{eqn: split beta prob}, we then obtain
\begin{align*}
  &\begin{aligned}
\limsup_{n \to \infty}
\prob\left\{\phi_{T_{n}}(v) \leq \phi_{T_{n}}(1)\right\} 
&\leq 
6 \cdot 2^{\frac{\ell(v)}{4}} \exp\left(-\frac{t}{4}\right)
\\ &\qquad + \exp\left(\frac{\ell(v)}{2} - \frac{t}{\ell(v)(d - 2)} - \frac{\ell(v)}{2} \log \left(\frac{\ell(v)^{2}(d-2)}{2t}\right) \right).
  \end{aligned}
\end{align*}
Taking \(t = 2 \ell(v) \log(\ell(v)),\) we have
\begin{align*}
  &\begin{aligned}
\limsup_{n \to \infty}
\prob\left\{\phi_{T_{n}}(v) \leq \phi_{T_{n}}(1)\right\}  
&\leq 
6 \exp\left( -\frac{\ell(v)}{2} \left(\log(\ell(v)) - \frac{\log(2)}{2}\right) \right)
\\& \qquad 
+ \exp\left(-\frac{\ell(v)}{2} \log\left(\frac{1}{e}\right) - \frac{\ell(v)}{2} \log\left(\frac{\ell(v)(d-2)}{4 \log(\ell(v))}\right)\right) \\ 
&\leq 
6 \exp\left(-\frac{\ell(v)}{2} \log\left(\frac{\ell(v)}{2}\right) \right) 
\\ &\qquad 
+ \exp\left(-\frac{\ell(v)}{2} \log\left(\frac{\ell(v)(d-2)}{4e \log(\ell(v))}\right) \right),
  \end{aligned}
\end{align*}
which completes the proof.


\section{Conclusion and Future Directions}
\label{SecDiscussion}

In this paper, we have provided confidence sets for diffusions on regular trees. In particular, we have provided bounds on the probability of error as a function of the size of the estimating sets selected by two diffusion source estimators, leading to upper bounds on the number of nodes required to guarantee that the source node lies in the resulting confidence sets. The main future research direction is to find confidence sets and corresponding bounds on error probabilities for more general trees and graphs, which would be more representative of the complicated topologies arising in real-world networks. Our example of asymmetry shows that a better estimator should take into account more aspects of the underlying graph structure, including degree inhomogeneity.

Another avenue of research is to find fast, local algorithms for obtaining confidence sets for the source of a diffusion, similar to the works of Borgs et al.~\cite{borgs} and Frieze and Pegden~\cite{frieze}. In practice, one might not have access to the entire set of infected nodes and/or the entire topology of the underlying network. In cases where the number of nodes \(n\) in the diffusion is very large, an \(O(n)\) algorithm for estimating the diffusion source could be prohibitively slow, and a faster algorithm would be preferable.

\section*{Acknowledgments}

We acknowledge Robin Pemantle for helpful comments and discussions.


\bibliography{networks}

\begin{thebibliography}{10}

\bibitem{And82}
R.~Anderson.
\newblock {\em The Population Dynamics of Infectious Diseases: Theory and
  Applications}.
\newblock Population and Community Biology Series. Springer US, 1982.

\bibitem{athreya1969}
K.~Athreya.
\newblock On a characteristic property of {P}olya's urn.
\newblock {\em Studia Scientiarum Mathematicarum Hungarica}, (4):31--35, 1969.

\bibitem{Bai75}
N.~Bailey.
\newblock {\em The Mathematical Theory of Infectious Diseases and its
  Applications}.
\newblock Griffin, London, 1975.

\bibitem{barabasi1999}
A.-L. Barab\'{a}si and R.~Albert.
\newblock Emergence of scaling in random networks.
\newblock {\em Science}, 286:509--512, 1999.

\bibitem{borgs}
C.~Borgs, M.~Brautbar, J.~Chayes, S.~Khanna, and B.~Lucier.
\newblock The power of local information in social networks.
\newblock {\em Proceedings of the 8th International Workshop on Internet and
  Network Economics}, pages 406--419, 2012.

\bibitem{brautbar}
M.~Brautbar and M.~Kearns.
\newblock Local algorithms for finding interesting individuals in large
  networks.
\newblock {\em Innovations in Theoretical Computer Science}, 2010.

\bibitem{bubeck}
S.~Bubeck, L.~Devroye, and G.~Lugosi.
\newblock Finding {A}dam in random growing trees.
\newblock {\em Random Structures and Algorithms}, page to appear, 2015.

\bibitem{frieze}
A.~Frieze and W.~Pegden.
\newblock Looking for vertex number one.
\newblock {\em arXiv preprint arXiv:1408.6821}, 2014.

\bibitem{fuchs}
M.~Fuchs and P.-D. Yu.
\newblock Rumor source detection for rumor spreading on random increasing
  trees.
\newblock {\em Electronic Communications in Probability}, 20(2), 2015.

\bibitem{Jac08}
M.~O. Jackson.
\newblock {\em Social and Economic Networks}.
\newblock Princeton University Press, Princeton, NJ, USA, 2008.

\bibitem{karamchandani}
N.~Karamchandani and M.~Franceschetti.
\newblock Rumor source detection under probabilistic sampling.
\newblock {\em IEEE International Symposium on Information Theory}, pages
  2184--2188, 2013.

\bibitem{KerMcK27}
W.~O. Kermack and A.~G. McKendrick.
\newblock A contribution to the mathematical theory of epidemics.
\newblock {\em Proceedings of the Royal Society of London A: Mathematical,
  Physical and Engineering Sciences}, 115(772):700--721, 1927.

\bibitem{mau92}
D.~R. Mauldin, W.~D. Sudderth, and S.~Williams.
\newblock Polya trees and random distributions.
\newblock {\em The Annals of Statistics}, 20(3):1203--1221, 1992.

\bibitem{Mor93}
M.~Morris.
\newblock Epidemiology and social networks: {M}odeling structured diffusion.
\newblock {\em Sociological Methods and Research}, 22(1):99--126, 1993.

\bibitem{New02}
M.~E.~J. Newman.
\newblock Spread of epidemic disease on networks.
\newblock {\em Physical review E}, 66(1), 2002.

\bibitem{Ore65}
O.~Ore.
\newblock {\em Theory of Graphs}.
\newblock American Mathematical Society colloquium publications. American
  Mathematical Society, 1965.

\bibitem{PasVes01}
R.~Pastor-Satorras and A.~Vespignani.
\newblock Epidemic spreading in scale-free networks.
\newblock {\em Phys. Rev. Lett.}, 86:3200--3203, Apr 2001.

\bibitem{Rap53}
A.~Rapoport.
\newblock Spread of information through a population with socio-structural
  bias: {I. A}ssumption of transitivity.
\newblock {\em The bulletin of mathematical biophysics}, 15(4):523--533, 1953.

\bibitem{shah}
D.~Shah and T.~Zaman.
\newblock Rumors in a network: {W}ho's the culprit?
\newblock {\em IEEE Transactions on Information Theory}, 2011.

\bibitem{shah_universal}
D.~Shah and T.~Zaman.
\newblock Rumor centrality: {A} universal source detector.
\newblock {\em Sigmetrics}, 2012.

\bibitem{Stro01}
S.~H. Strogatz.
\newblock Exploring complex networks.
\newblock {\em Nature}, 410(6825):268--276, 2001.

\bibitem{zhu2014}
K.~Zhu and L.~Ying.
\newblock A robust information source estimator with sparse observations.
\newblock {\em Computational Social Networks}, 1(1), 2014.

\end{thebibliography}
\bibliographystyle{abbrv}

\newpage 
\begin{appendices}

\section{Supporting Lemmas for Theorem~\ref{thm: max subtree regular tree}}

In this appendix, we provide the proofs of the lemmas used to prove Theorem~\ref{thm: max subtree regular tree}.


\subsection{Proof of Lemma~\ref{LemDirichlet}}
\label{AppLemDirichlet}

Let \(E_{i, K}^{n}\) denote the number of edges between 
\(T_{i, K}^{n}\) and \(\mathcal{G} \setminus T_n\).
The key observation is that
\((E_{1, K}^{n}, \ldots, E_{K, K}^{n})\) evolves according to a P\'{o}lya urn with replacement matrix 
\((d - 2) I_{K}\).
Thus, we have the convergence in distribution
\begin{equation}
\frac{1}{n(d - 2) + 2}(|E_{1, K}^{n}|, \dots, |E_{K, K}^{n}|)
\longrightarrow
\dirichlet\left(\frac{d_{\GG}(1) - d_{T_{K}}(1)}{d - 2},
\ldots, \frac{d_{\GG}(K) - d_{T_{K}}(K)}{d - 2}\right),
\label{eqn: dirichlet convergence}
\end{equation}
as \(n \to \infty\), where we exclude component \(i\) from the above vectors if 
\(d_{\GG}(i) - d_{T_{K}}(i) = 0\).
Also note that we may write
\[
E_{i, K}^{n} 
= (d - 2) |T_{i, K}^{n}| - d_{T_{K}}(i) + 2,
\]
which implies that 
\[\
\lim _{n \to \infty} \frac{1}{n}|T_{i, K}^{n}| 
=
\lim _{n \to \infty} \frac{1}{n(d - 2) + 2} E_{i, K}^{n}.
\]
Thus, the limit of 
\(\frac{1}{n}(|T_{1, K}^{n}|, \ldots, |T_{K, K}^{n}|)\)
is the same Dirichlet random variable as in
equation~\eqref{eqn: dirichlet convergence}.


\subsection{Proof of Lemma~\ref{LemWtail}}
\label{AppLemWtail}

From Lemma~\ref{LemDirichlet}, we see that
\[
\frac{1}{n}(W_{k, K}^{n}, |T_{k, K}^{n}|)
\to 
\dirichlet\left(\frac{K(d-2) + 2 - (d_{\mathcal{G}}(k) - d_{T_{K}}(k))}{d - 2},  \frac{d_{\mathcal{G}}(k) - d_{T_{K}}(k)}{d - 2} \right).
\] 
For simplicity, denote the parameters of the Dirichlet distribution by \(c_{1, k}\) and \(c_{2, k}\).
Of course, this simply means that \(W_{k, K}^{n} / n\) has a beta distribution unless
\(d_{\mathcal{G}}(k) = d_{T_{K}}(k)\). 
In this case, we see that 
\(W_{k, K}^{n} = n - 1\), so \(W_{k, K}^{n} > (1 - \eta)n\) when \(n > 1 / \eta\). In the general case, we may write
\begin{align}
\label{EqnW}
  &\begin{aligned}
  \lim_{n \to \infty}
\prob\left\{W_{k, K}^{n} \leq (1 - \eta)n \right\} 
&= 
\frac{1}{\beta \left(c_{1, k}, c_{2, k}\right)} 
\int_{0}^{1 - \eta} x^{c_{1, k} - 1} (1 - x)^{c_{2, k} - 1} \; dx \\ 
&\leq 
\frac{1}{\beta(c_{1, k}, c_{2, k})} \int_{0}^{1 - \eta} x^{c_{1, k} - 1} \; dx \\ 
&=
\frac{1}{c_{1, k} \beta(c_{1, k}, c_{2, k})} 
(1 - \eta)^{c_{1, k}} \\ 
&\leq
\frac{1}{(K  - 1) \beta(c_{1, k}, c_{2, k})} (1 - \eta)^{K - 1 + \frac{1}{d - 2}}.
  \end{aligned}
\end{align}
We now find a lower bound for $\beta(c_{1,k}, c_{2,k})$.
By Stirling's approximation,
\begin{equation}
\sqrt{\frac{2 \pi}{x}}\left(\frac{x}{e}\right)^{x}
\leq \Gamma(x)
\leq \sqrt{\frac{2 \pi}{x}}\left(\frac{x}{e}\right)^{x} e^{\frac{1}{12 x}}.
\label{eqn: stirling inequalities}
\end{equation}
Furthermore, the gamma function is bounded below by $\frac{7}{8}$ on the positive real numbers. Hence,
\begin{align*}
  &\begin{aligned}
& \beta(c_{1, k}, c_{2, k})^{-1}
= \frac{\Gamma\left(K - 1 + \frac{d}{d - 2}\right)}{\Gamma\left(K - 1 + \frac{d_{T_{K}}(k)}{d - 2}\right) \Gamma\left(\frac{d - d_{T_{K}}(k)}{d - 2}\right)} \\ 
& \quad \leq 
\frac{
\sqrt{\frac{2 \pi}{K - 1 + \frac{d}{d - 2}}} \left(\frac{K - 1 + \frac{d}{d - 2}}{e}\right)^{K - 1 + \frac{d}{d - 2}} e^{\frac{1}{12\left(K - 1 + \frac{d}{d - 2}\right)}}
}{
\frac{7}{8}
\sqrt{\frac{2 \pi}{K - 1 + \frac{d_{T_{K}}(k)}{d - 2}}} 
\left(
\frac{K - 1 + \frac{d_{T_{K}}(k)}{d - 2}}{e}\right)^{K - 1 + \frac{d_{T_{K}}(k)}{d - 2}}
} \\ 
& \quad = 
\frac{8}{7}
\sqrt{\frac{K - 1 + \frac{d_{T_{K}}(k)}{d - 2}}{K - 1 + \frac{d}{d - 2}}}
e^{\frac{1}{12\left(K - 1 + \frac{d}{d - 2}\right)} - \frac{d - d_{T_{K}}(k)}{d - 2}} \frac{\left(K - 1 + \frac{d}{d - 2}\right)^{K - 1 + \frac{d}{d - 2}}}{\left(K - 1 + \frac{d_{T_{K}}(k)}{d - 2}\right)^{K - 1 + \frac{d_{T_{K}}(k)}{d - 2}}}\\ 
& \quad \le
\frac{8}{7} e^{[12(K - 1)]^{-1}}
\frac{\left(K - 1 + \frac{d}{d - 2}\right)^{K - 1 + \frac{d}{d - 2}}}{\left(K - 1\right)^{K - 1 + \frac{1}{d-2}}}.
  \end{aligned}
\end{align*}
We now examine the fraction in the last expression.
Denoting it by \(A\), we have
\begin{align*}
  &\begin{aligned}
A & =
\left(1 + \frac{\frac{d}{d - 2}}{K - 1}\right)^{K - 1}
\frac{(K + 1)^{\frac{d}{d - 2}}}{(K - 1)^{\frac{1}{d - 2}}} \\ 
& \le e^{\frac{d}{d - 2}} 
\frac{(K + 1)^{1 + \frac{2}{d - 2}}}{(K - 1)^{\frac{1}{d - 2}}}. 
  \end{aligned}
\end{align*}
In particular, we have
\[
\beta(c_{1, k}, c_{2, k})^{-1}
\leq C K^{1 + \frac{1}{d - 2}},
\]
for an appropriate constant \(C\). Plugging this result back into the bound~\eqref{EqnW}, we therefore obtain
\begin{align*}
  &\begin{aligned}
  \lim_{n \to \infty}
\prob\left\{W_{k, K}^{n} \leq (1 - \eta)n \right\}  
&\leq \frac{CK^{1 + \frac{1}{d - 2}}}{K - 1} (1 - \eta)^{K - 1 + \frac{1}{d - 2}} \\ 
&\leq C K^{1 + \frac{1}{d - 2}} (1 - \eta)^{K - 1 + \frac{1}{d - 2}},
  \end{aligned}
\end{align*}
which completes the proof.


\section{Supporting Lemmas for Theorem~\ref{thm: inv rumor centrality distance}}

In this appendix, we provide the proofs of the lemmas used to establish Theorem~\ref{thm: inv rumor centrality distance}.


\subsection{Proof of Lemma~\ref{LemBetaConv}}
\label{AppLemBetaConv}

Let \(E_{v}^{n}\) denote the number of edges in \(\GG \setminus (T_{n}, 1)_{v \downarrow}\).
We may write
\[
E_{v}^{n}
= (d - 2) |(T, 1)_{v \downarrow}| + 1,
\]
for \(v > 1\), and 
\[
E_{1}^{n}
= (d - 2)|T_{n}| + 2.
\] 
In particular, we see that 
\[
\lim_{n \to \infty}
\frac{|(T_{n}, 1)_{v \downarrow}|}{n}
=
\lim_{n \to \infty}
\frac{|E_{v}^{n}|}{(d - 2)n + 2}.
\]
Furthermore, note that 
\(\left(E_{(1)}^{n}, \ldots, E_{(d -1)}^{n}\right)\)
evolves according to a P\'{o}lya urn with replacement matrix \((d - 2)I_{d}\).
This means that
\[
\left(\frac{E_{(1)}^{n}}{(d - 2)n + 2}, \ldots, \frac{E_{(d-1)}^{n}}{(d - 2)n + 2}\right)
\longrightarrow 
\dirichlet\left(\frac{1}{d - 2}, \ldots, \frac{1}{d - 2}\right).
\]
Moreover, we may modify the analysis slightly to deduce the behavior of \(|E_{(j_{1}, \ldots, j_{i})}^{n}|\).
Let \(N_{(j_{1}, \ldots, j_{k})}^{n}\) denote the number of descendants of \((j_{1}, \ldots, j_{k})\) in the diffusion subtree \(T_{n}\).
Note that
\[
(d - 1) + (d - 2) N_{(j_{1}, \ldots, j_{k})}^{n}
= \sum_{i = 1}^{d - 1} E_{(j_{1}, \ldots, j_{k}, i)}^{n}
= E_{(j_{1}, \ldots, j_{k})}^{n}.
\]
Furthermore, \((E_{(j_{1}, \ldots, j_{k}, 1)}^{n}, \ldots, E_{(j_{1}, \ldots, j_{k}, d - 1)}^{n})\)
evolves according to a P\'{o}lya urn with replacement matrix 
\((d - 2) I_{d - 1}\), implying that
\begin{multline}
\left(\frac{E_{(j_{1}, \ldots, j_{k}, 1)}^{n}}{(d - 2)N_{(j_{1}, \ldots, j_{k})}^{n} + (d-1) }, 
\ldots, 
\frac{E_{(j_{1}, \ldots, j_{k}, d - 1)}^{n}}{(d - 2)N_{(j_{1}, \ldots, j_{k})}^{n} + (d-1) }\right) \\
\longrightarrow \dirichlet\left(\frac{1}{d - 2}, \ldots, \frac{1}{d - 2}\right),
\end{multline}
as \(N_{(j_{1}, \ldots, j_{k})}^{n} \to \infty\).
From this, we see that
\begin{align*}
\lim_{n \to \infty} \frac{E_{(j_{1}, \ldots, j_{\ell(v)})}^{n}}{(d-2)n} 
&=
\lim_{n \to \infty}
\left(\frac{E_{(j_{1})}^{n}}{(d-2)n + 2} \right) 
\cdot \left(\frac{E_{(j_{1}, j_{2})}^{n}}{E_{(j_{1})}^{n}} \right)
\cdots \left(\frac{E_{(j_{1}, \ldots, j_{\ell(v)})}^{n}}{E_{(j_{1}, \ldots, j_{\ell(v) - 1})}^{n}} \right) \\
& = \prod_{i = 1}^{\ell(v)} B_{(j_{1}, \ldots, j_{\ell(v)})}.
\end{align*}

This does not establish the independence of the random variables appearing in the product.
The independence is intuitively obvious, since the proportion of nodes in \(E^{n}_{(j_{1})}\) relative to the rest of the diffusion does not affect the proportion of nodes in \(E^{n}_{(j_{1}, j_{2})}\) relative to \(E^{n}_{(j_{1})}\).
A similar result is proved in Mauldin et al.~\cite{mau92}, but the proof uses a theorem of de Finetti.
We provide a more direct proof here.
To minimize notation, we consider the case \(\ell(v) = 2\), since the result may easily be extended via induction to products of arbitrary length.

We want to show that
\begin{equation}
\prob\left\{B_{(j_{1})} \leq x, \; B_{(j_{1}, j_{2})} \leq y\right\}
= 
\prob\left\{B_{(j_{1})} \leq x\right\}
\prob\left\{B_{(j_{1}, j_{2})} \leq y\right\}.
\label{eqn: ind to prove}
\end{equation}

First, define the stopping times \(\tau_{i}\) by
setting \(\tau_{0} = \min\{n: E^{n}_{(j_{1})} = d - 1\}\) and
\[
\tau_{i} 
= \min\left\{n > \tau_{i - 1}: E^{n}_{(j_{1})} > E^{n - 1}_{(j_{1})}\right\}.
\]
Thus, the \(\tau_{i}\) are the times at which nodes are added to the subtree \((T_{n}, 1)_{(j_{1}) \downarrow}.\)

Define the random vector
\begin{equation*}
R_{n} = \left\{E^n_{(j_{1}', \ldots, j_{m}')}: m \geq 0, \text{ and } (j_{1}', j_{2}') \neq (j_{1}, j), \text{ for any } j = 1, \ldots, d - 1\right\}.
\end{equation*}
In words, \(R_{n}\) captures the state of the tree at time \(n\), excluding knowledge of the subtrees of \((T_{n}, 1)_{(j_{1}) \downarrow}\).
We can now define the \(\sigma\)-field 
\[
\ff_{R, n} = \sigma\left(R_{1}, \ldots, R_{n}\right).
\]
We will be interested in \(\ff_{R, \tau_{k}}\).
Note that the events which are measurable with respect to \(\ff_{R, \tau_{k}}\) are those in which \(k\) nodes have been added to the subtree of \((T_{n}, 1)_{(j_{1}) \downarrow}\), so that
\(|(T_{n}, 1)_{(j_{1}) \downarrow}| = k + 1\),
and the growth of the diffusion subtree \(T_{n}\) may be specified in any arbitrary manner outside of \((T_{n}, 1)_{(j_{1})}.\)

Now define the random variable 
\[
Y_{n} 
= \frac{E^{n}_{(j_{1}, j_{2})}}{E^{n}_{(j_{1})}},
\]
and define the \(\sigma\)-field
\[
\ff_{Y, \tau_{k}}
= \sigma\left(Y_{\tau_{1}}, \ldots, Y_{\tau_{k}}\right).
\]
The events that are measurable with respect to \(\ff_{Y, \tau_{k}}\) are exactly the events which specify the values of \(Y_{\tau_{i}}\), for \(i \leq k\).

Hence, we have defined \(\sigma\)-fields for events concerning the growth of \(T_{n}\) outside the subtrees 
\((T_{n}, 1)_{(j_{1}, j) \downarrow}\), for \(j = 1, \ldots, d - 1\),
and events concerning the growth of \((T_{n}, 1)_{(j_{1}, j_{2})\downarrow}\) relative to \((T_{n}, 1)_{(j_{1})\downarrow}\).
Now consider \(A \in \ff_{R, \tau_{k}}\) and \(B \in \ff_{Y, \tau_{k}}\).
We want to show that \mbox{\(\prob\left(A \cap B\right) = \prob\left(A\right) \prob\left(B\right)\),}
since equation~\eqref{eqn: ind to prove} then follows easily.

Suppose \(A\) and \(B\) are nonempty.
We can view elements of \(A\) and \(B\) as tree histories \(H\) in the sample space \(\Omega\) in the following sense:
Let 
\[
H_{n}
= (E^{n}_{\emptyset}, E^{n}_{(1)}, \ldots, E^{n}_{(d - 1)}, E^{n}_{(1, 1)}, \ldots)
\]
be the state of the tree at time \(n\).
A history \(H\) is a sequence \(H = (H_{1}, H_{2}, \ldots)\)
such that it is possible to obtain the tree \(T_{n + 1}(H)\) from \(T_{n}(H)\), if we consider \(T_{n + 1}\) and \(T_{n}\) as functions of histories, by adding the appropriate vertex to the diffusion.
Thus, the measure \(\prob\) assigns to sets of histories some probability according to the evolution of the diffusion process.

Now suppose \(H\) is a history in \(A\).
Since \(A \in \ff_{R, \tau_{k}}\), 
there exist \(k\) times \(\nu_{1}, \ldots, \nu_{k}\) at which a node is added to \((T_{n}, 1)_{(j_{1}) \downarrow}\).
First, note that \(A\) contains all histories \(H'\) such that \(H_{i}' = H_{i}\), for \(i = 1, \ldots, \nu_{k}\),
since the growth of the tree after the addition of the \(k^\text{th}\) node is non-measurable with respect to \(\ff_{R, \tau_{k}}\).
Consequently, we define
\[
C_{H} 
= \{H' : H_{i}' = H_{i}, \; i = 1, \ldots, \nu_{k}\},
\]
the set of histories which may differ from $H$ only after time $\nu_k$.

Second, observe that \(A\) must also contain every history \(H''\) such that
\begin{equation}
R(H''_{i}) = R(H_{i}),
\label{eqn: er1}
\end{equation} 
for \(i = 1, \ldots, \nu_{k}\), where we view the vector $R$ of subtree counts defined above as a function of histories
In other words, \(H''\) agrees with \(H'\) except possibly on the subtrees of \((T_{n}, 1)_{(j_{1})\downarrow}\).
Note that equation~\eqref{eqn: er1} defines an equivalence relation between histories.
In particular, we may partition the set \(A\) into equivalence classes.
Let \(\ppp\) denote the partition and \(\ccc\) denote an equivalence class, so that
\[
A 
= \bigsqcup_{\ccc \in \ppp} \ccc.
\]

We may similarly partition \(B\) into sets of histories.
Let \(H\) be a history in \(B\).
Again, using \(\nu_{1}, \ldots, \nu_{k}\) to denote the times when nodes are added to subtrees of \((T_{n}, 1)_{(j_{1})\downarrow}\), we see that
the set \(B\) must contain \(C_{H}\). Furthermore, \(B\) must contain any history \(H'\) with the following property:
Let \(\nu'_{1} \ldots, \nu'_{k}\) denote the times when nodes are added to subtrees of \((T_{n}, 1)_{(j_{1})\downarrow}\) in \(H'\).
Then \(H'\) is in \(B\) if
\begin{equation}
Y_{\nu'_{i}}(H')
= 
Y_{\nu_{i}}(H),
\label{eqn: er2}
\end{equation}
for \(i = 1, \ldots, k\).
It is easy to see that equation~\eqref{eqn: er2} also defines an equivalence relation and partitions \(B\) into equivalence classes.
Hence, we may write
\[
B = \bigsqcup_{\ddd \in \qqq} \ddd.
\]
In particular, we have
\begin{align*}
  &\begin{aligned}
A \cap B 
&= 
\left(\bigsqcup_{\ccc \in \ppp} \ccc \right) \cap  \left(\bigsqcup_{\ddd \in \qqq} \ddd \right) \\ 
&= \bigsqcup_{\ccc \in \ppp} \bigsqcup_{\ddd \in \qqq} ( \ccc \cap \ddd ).
  \end{aligned}
\end{align*}
It now suffices to show that
\(\prob( \ccc \cap \ddd ) = \prob(\ccc) \prob(\ddd)\), 
since additivity then implies that we can obtain
\mbox{\(\prob(A \cap B) = \prob\left(A\right) \prob\left(B\right)\).}
Pick some set \(\ccc \cap \ddd\) and some \(H\) in \(\ccc \cap \ddd\), and let $\nu_1, \dots, \nu_k$ denote the times when a node is added to the subtree
\((T_{n}, 1)_{(j_{1})\downarrow}\) in $H$.

The value of \(Y_{\nu_{i}}(H)\) must be fixed for each \(i \leq k\), and all histories \(H'\) with \mbox{\(\nu'_{i} = \nu_{i}\)} and 
\mbox{\(Y_{\nu'_{i}}(H') = Y_{\nu_{i}}(H)\)} must also lie in \( \ccc \cap \ddd\).
In fact, \(H\) is only permitted to vary after \(\nu_{k}\), on the subtrees \((T_{n}, 1)_{(j_{1}, j'_{2},  \ldots, j'_{m})\downarrow}\),
for \(m \geq 2\) and \(j'_{2} \neq j_{2}\), and on the subtrees \((T_{n}, 1)_{(j_{1}, j_{2}, j'_{3}, \ldots, j'_{m})\downarrow}\), for \(m \geq 3\).

Now, we simply realize that \(\ccc \cap \ddd\) specifies \((R_{i}, Y_{i})\), \(\ccc\) specifies \(R_{i}\), and \(\ddd\) specifies \(Y_{i}\).
Furthermore, \(R_{i}\) depends only on \(R_{i - 1}\) and \(Y_{\tau_{i}}\) depends only on \(Y_{\tau_{i - 1}}\). It follows that $\prob(\ccc \cap \ddd) = \prob(\ccc) \prob(\ddd)$, implying that \(\prob\left(A \cap B\right) = \prob\left(A\right) \prob\left(B\right)\), as well.
Hence, \(\ff_{R, \tau_{k}}\) and \(\ff_{Y, \tau_{k}}\) are independent.

In particular, if we let \(B_{(j_{1})}\) and \(B_{(j_{1}, j_{2})}\) denote the limiting beta-distributed random variables, we have
\begin{align*}
  &\begin{aligned}
\prob\left\{B_{(j_{1})} \leq x, \; B_{(j_{1}, j_{2})} \leq y\right\} 
&=
\lim_{k \to \infty} 
\prob\left\{\frac{E^{\tau_{k}}}{(d - 2)n + 2} \leq x, \; Y_{\tau_{k}} \leq y\right\} \\ 
&=
\lim_{k \to \infty}
\prob\left\{\frac{E^{\tau_{k}}}{(d - 2)n + 2} \leq x\right\} 
\prob\left\{Y_{\tau_{k}} \leq y\right\} \\ 
&=
\prob\left\{B_{(j_{1})} \leq x\right\}  
\prob\left\{B_{(j_{1}, j_{2})} \leq y\right\}. 
  \end{aligned}
\end{align*}
Thus, the two random variables are independent, which is what we wanted to show.



\subsection{Proof of Lemma~\ref{LemFirst}}
\label{AppLemFirst}

We use the fact that \(\frac{1}{2}\min(x, 1) \leq 1 - \exp(-x)\), for \(x \geq 0\).
This gives
\begin{align*}
  &\begin{aligned}
\frac{1}{2^{\ell(v)}} 
\prod_{i = 1}^{\ell(v)} \min\left(-\sum_{k = 1}^{i} \log(B_{(j_{1}, \ldots, j_{k})}' ), \; 1\right)
&\leq
\prod_{i = 1}^{\ell(v)} \left(1 - \prod_{k = 1}^{i} B_{(j_{1}, \ldots, j_{k})}'\right), 
  \end{aligned}
\end{align*}
from which we obtain
\[
\frac{1}{2^{\ell(v)}} B
\leq 
\prod_{i = 1}^{\ell(v)} \left(1 - \prod_{k = 1}^{i} B_{(j_{1}, \ldots, j_{k})}'\right), 
\]
where 
\(B = \prod_{i = 1}^{\infty}  \min\left(-\sum_{k = 1}^{i} \log(B_{(j_{1}, \ldots, j_{k})}' ), \; 1\right).\)

Combining this bound with Lemma \ref{lemma: beta bound} in Appendix~\ref{AppAuxiliary}, we then have
\begin{align*}
\begin{aligned}
\prob\left\{\prod_{i = 1}^{\ell(v)} \left(1 - \prod_{k = 1}^{i} B_{(j_{1}, \ldots, j_{k})}'\right) 
\leq 
\exp(-t) \right\}
&\leq 
\prob\left\{\frac{1}{2^{\ell(v)}} B \leq \exp(-t)\right\} \\ 
&\leq 
6 \cdot 2^{\frac{\ell(v)}{4}} \exp\left(-\frac{t}{4}\right),
\end{aligned}
\end{align*}
as desired.


\subsection{Proof of Lemma~\ref{LemSecond}}
\label{AppLemSecond}

Observe that
\begin{align}
\label{EqnChernoff}
  &\begin{aligned}
\prob\left\{\exp(-t)
\leq 
\prod_{i = 1}^{\ell(v)} \prod_{k = 1}^{i} B_{(j_{1}, \ldots, j_{k})}'\right\}
&=
\prob\left\{
-\sum_{i = 1}^{\ell(v)} \sum_{k = 1}^{i} \log( B_{(j_{1}, \ldots, j_{k})}' )
\leq t
\right\} \\  
&=  
\prob\left\{
-\sum_{i = 1}^{\ell(v)} [1 + \ell(v) - i]  \log(B_{(j_{1}, \ldots, j_{i})}' )
\leq t
\right\}.
  \end{aligned}
\end{align}
We now use a Chernoff bound. Let $P$ denote the expression on the right-hand side of inequality~\eqref{EqnChernoff}, and let $\ell = \ell(v)$. For $\lambda > 0$, we have


\begin{align*}
  &\begin{aligned}
P
&\leq 
\expect \exp\left(\lambda t - \left(-\sum_{i = 1}^{\ell} \lambda[1 + \ell - i]  \log(B_{i})\right)\right) \\ 
&= 
\exp\left(\lambda t\right)
\prod_{i = 1}^{\ell} 
\expect\exp\left(\lambda [1 + \ell - i]  \log(B_{i})\right) \\
&= 
 \exp\left(\lambda t\right)
\prod_{i = 1}^{\ell}  
\expect\left[B_{i}^{\lambda[1 + \ell - i]}\right].
  \end{aligned}
\end{align*}
Now, we compute the moments as
\begin{align*}
  &\begin{aligned}
\expect\left[B_{i}^{\lambda[1 + \ell - i]}\right]
&= \frac{1}{\beta\left(\frac{1}{d - 2}, 1\right)}  \int_{0}^{1} x^{\lambda[1 + \ell - i] + \frac{1}{d - 2} - 1} \; dx\\ 
&= \frac{1}{d - 2} \left(\frac{1}{\lambda[1 + \ell - i] + \frac{1}{d - 2}}\right) \\ 
&= \frac{1}{\lambda[1 + \ell - i](d - 2) + 1}.
  \end{aligned}
\end{align*}
Thus, we have
\[
P
\leq 
\exp\left(\lambda t
- \sum_{i = 1} \log\left(\lambda[1 + \ell - i](d - 2) + 1\right)
\right).
\]
Using the fact that \(\log(x + 1) / x\) is nonincreasing when \(x \geq 0\), we may further write
\begin{align*}
  &\begin{aligned}
P 
& \le
\exp\left(
\lambda t - \sum_{i = 1}^{\ell} \lambda[1 + \ell - i](d - 2) 
\frac{\log(\lambda[1 + \ell - i](d - 2) + 1)}{\lambda[1 + \ell - i](d - 2)}
\right) \\ 
&\leq 
\exp\left(
\lambda t 
- \sum_{i = 1}^{\ell} \lambda[1 + \ell - i](d - 2) \frac{\log(\lambda \ell (d - 2) + 1)}{\lambda \ell (d - 2)}
\right) \\ 
&\leq 
\exp\left(
\lambda t - \frac{1}{\ell + 1} \log \left(\lambda \ell (d - 2) + 1)\right) 
\sum_{i = 1}^{\ell} [1 + \ell - i]
\right) \\ 
&= 
\exp\left(
\lambda t - \frac{\ell}{2} \log \left(\lambda \ell (d - 2) + 1 \right)
\right).
  \end{aligned}
\end{align*}
Let \(g(\lambda)\) denote the logarithm of the last expression as a function of \(\lambda\).
We now find \(\lambda > 0\) to minimize \(g(\lambda)\).
Setting \(g'(\lambda^{*}) = 0\) and solving, we obtain
\[
\lambda^{*}
= 
\frac{\ell}{2t} - \frac{1}{\ell (d - 2)}.
\]
Since \(g''(\lambda) > 0\) for all \(\lambda > 0\),
we see that \(\lambda^{*}\) minimizes \(g\).
Substituting this into our bound on \(P\), we then have
\begin{align*}
  &\begin{aligned}
P 
&\leq 
\exp\left(
\frac{\ell}{2} - \frac{t}{\ell (d - 2)} 
- \frac{\ell}{2} \log\left(\frac{\ell^{2}(d-2)}{2t} \right)
\right),
  \end{aligned}
\end{align*}
which completes the proof of the lemma.



\section{Proofs of Corollaries}
\label{AppCors}

In this appendix, we provide the proofs of Corollaries~\ref{CorConfidence} and~\ref{CorPhiEst}.

\subsection{Proof of Corollary~\ref{CorConfidence}}
\label{AppCorConfidence}

From the bound of Theorem~\ref{thm: max subtree regular tree}, it suffices to choose $\eta \in (0,1)$ and $K$ large enough such that
\begin{equation*}
C_1 \eta^{1 + \frac{1}{d-2}} + C_2 K^{2 + \frac{1}{d-2}} (1-\eta)^{K-1+\frac{1}{d-2}} \le \epsilon.
\end{equation*}
Suppose $\epsilon < 2C_1$. We will take $\eta = \left(\frac{\epsilon}{2C_1}\right)^{\frac{d-2}{d-1}}$, so $C_1 \eta^{1 + \frac{1}{d-2}} \le \frac{\epsilon}{2}$. Taking logarithms, it now suffices to show that
\begin{equation*}
\left(2+ \frac{1}{d-2}\right) \log K + \left(K - 1 + \frac{1}{d-2}\right) \log\left(1 - \left(\frac{\epsilon}{2C_1}\right)^{\frac{d-2}{d-1}}\right) \le \log\left(\frac{\epsilon}{2C_2}\right).
\end{equation*}
Clearly, this holds provided
\begin{equation}
\label{EqnKsplit}
3\log K - \frac{K}{2} \left(\frac{\epsilon}{2C_1}\right)^{\frac{d-2}{d-1}} \le \log \left(\frac{\epsilon}{2C_2}\right),
\end{equation}
since $d \ge 3$ and $\log(1-x) \le -x$ for $x \in (0,1)$. We will now show how to choose $K$ such that
\begin{subequations}
\begin{equation}
\label{EqnK1}
3 \log K \le \frac{K}{4} \left(\frac{\epsilon}{2C_1}\right)^{\frac{d-2}{d-1}},
\end{equation}
and
\begin{equation}
\label{EqnK2}
\log\left(\frac{2C_2}{\epsilon}\right) \le \frac{K}{4} \left(\frac{\epsilon}{2C_1}\right)^{\frac{d-2}{d-1}}.
\end{equation}
\end{subequations}
Clearly, inequality~\eqref{EqnKsplit} will follow from combining the bounds~\eqref{EqnK1} and~\eqref{EqnK2}. Using the fact that $\log K \le K^{\frac{1}{d-1}}$ for sufficiently large $K$, it is easy to see that inequality~\eqref{EqnK1} holds provided \mbox{$K \ge \frac{2C_1}{\epsilon} \cdot 12^{\frac{d-1}{d-2}}$}. Furthermore, using the fact that $\log\left(\frac{2C_2}{\epsilon}\right) \le \left(\frac{2C_2}{\epsilon}\right)^{\frac{1}{d-1}}$ for sufficiently small $\epsilon$, inequality~\eqref{EqnK2} holds provided $K \ge \frac{4}{\epsilon} \cdot (2C_1)^{\frac{d-2}{d-1}} (2C_2)^{\frac{1}{d-1}}$. We thus arrive at the desired result.


\subsection{Proof of Corollary~\ref{CorPhiEst}}
\label{AppCorPhiEst}

Suppose there is a vertex \(v\) such that \(\ell(v) > L\) and \(\phi_{T_{n}}(v) \leq \phi_{T_{n}}(1)\).
It suffices to bound the probability that \(\phi_{T_{n}}(v) \leq \phi_{T_{n}}(1)\), since if this happens with low probability,
then the probability that \(v\) is the rumor center is also low. 
Our plan is to use a sort of monotonicity to consider vertices far from the source at some fixed distance, rather than every distant vertex. We begin with the following lemma:

\begin{lem}
Let \(u\) and \(v\) be vertices such that \(u \neq 1\) and \(v\) is in the subtree \((T, 1)_{u \downarrow}\).
If \(\phi_{T}(v) \leq \phi_{T}(1),\) we also have \(\phi_{T}(u) \leq \phi_{T}(1)\).
\label{lemma: monotonicity of rc}
\end{lem}

\begin{proof}
Using our alternative notation, let \(v = (j_{1}, \ldots, j_{\ell(v)})\). We will show that
\begin{equation}
\label{EqnMonotone}
\varphi_T(j_1, \dots, j_k) \le \varphi_T(1), \qquad \forall k \le \ell(v).
\end{equation}
We proceed by induction. Note that the claim holds for $k = \ell(v)$ by assumption. Now suppose the claim holds for $k = K$, where $K \le \ell(v)$. By equation~\eqref{eqn: phiv condition}, this implies
\[ 
\prod_{i = 1}^{k} \frac{n - |(T, 1)_{(j_{1}, \ldots, j_{i}) \downarrow}|}{|(T, 1)_{(j_{1}, \ldots, j_{i}) \downarrow}|}
\leq 1.
\] 
For simplicity, call these multiplicands \(f(i).\)
The key observation is that as \(i\) increases, the subtree size \(|(T, 1)_{(j_{1}, \ldots, j_{i}) \downarrow}|\) decreases.
Thus, the function \(f(i)\) is strictly increasing.

First suppose \(f(K - 1) \leq 1.\) 
Then \(\prod_{i = 1}^{K - 1} f(i) \leq 1\), implying that 
\(\phi_{T}(j_1, \dots, j_{K-1}) \leq \phi_{T}(1)\).
On the other hand, if \(f(K - 1) \geq 1,\) then \(f(K) \geq 1\), also implying that
\[
\prod_{i = 1}^{K - 1} f(i) 
\leq 
\prod_{i = 1}^{K} f(i) 
\leq 1.
\]
This implies the claim~\eqref{EqnMonotone} and completes the induction.
\end{proof}


By Lemma \ref{lemma: monotonicity of rc},
there is a vertex \(u\) such that \(\ell(u) = L\), 
the vertex \(v\) is contained in \((T, 1)_{u \downarrow}\), and 
\(\phi_{T_{n}}(u) \leq \phi_{T_{n}}(1)\).
Hence,
\begin{multline*}
\limsup_{n \to \infty}
\prob\left\{\exists v: \ell(v) > L \text{ and } \phi(v) \leq \phi(1)\right\} \leq 
\prob\left\{\exists u: \ell(u) = L \text{ and } \phi(u) \leq \phi(1)\right\} \\
\leq 7|u: \ell(u) = L|
\exp\left(-\frac{L}{2} \log\left(\min\left\{\frac{L(d-2)}{4e \log(L)}, \frac{L}{2}\right\} \right)\right),
\end{multline*}
where the second inequality follows from Theorem~\ref{thm: inv rumor centrality distance}. Observing that 
\(|u: \ell(u) = L| \le d^{L}\)  and using some simple algebra completes the proof of inequality~\eqref{EqnConfWidth}.

Now, we derive our bound on the size of a $1-\epsilon$ confidence set. Clearly, it suffices to find $L$ sufficiently large such that the following two inequalities hold:
\begin{subequations}
\begin{equation}
\label{EqnL1}
7 \exp\left(-\frac{L}{2} \log \left(\frac{L}{4e d^2 \log(L)}\right) \right) \le \epsilon,
\end{equation}
\begin{equation}
\label{EqnL2}
7 \exp\left(-\frac{L}{2} \log\left(\frac{L}{2d^2}\right)\right) \le \epsilon.
\end{equation}
\end{subequations}
We first consider inequality~\eqref{EqnL1}, which is equivalent to
\begin{equation}
L \log\left(\frac{L}{4e d^2\log(L)}\right) \ge 2 \log\left(\frac{7}{\epsilon}\right).
\label{eqn: error eps}
\end{equation}
Using the fact that $\log L \le \sqrt{L}$, inequality~\eqref{eqn: error eps} is true provided $L
\geq 4 \frac{\log\left(\frac{7}{\epsilon}\right)}{\log\left(\frac{L}{C}\right)}$, or
\begin{equation*}
\frac{L}{C} \log\left(\frac{L}{C}\right) \ge 4 \log\left(\frac{7}{\epsilon}\right),
\end{equation*}
where \(C = 16 e^{2} d^4.\) We claim that this holds if
\begin{equation}
\label{EqnLCbd}
\frac{L}{C} \ge \frac{8 \log\left(\frac{7}{\epsilon}\right)}{\log \left(4 \log\left(\frac{7}{\epsilon}\right)\right)}.
\end{equation}
Indeed, under inequality~\eqref{EqnLCbd}, we have
\begin{equation*}
\frac{L}{C} \log\left(\frac{L}{C}\right) \ge \frac{8\log\left(\frac{7}{\epsilon}\right)}{\log\left(4 \log\left(\frac{7}{\epsilon}\right)\right)} \log\left(\frac{8 \log \left(\frac{7}{\epsilon}\right)}{\log\left(4 \log \left(\frac{7}{\epsilon}\right)\right)}\right).
\end{equation*}
Using the bound $\log x \le 2 \sqrt{x}$, with $x = 4 \log\left(\frac{7}{\epsilon}\right)$, we then have
\begin{equation*}
\frac{L}{C} \log\left(\frac{L}{C}\right) \ge \frac{8 \log\left(\frac{7}{\epsilon}\right)}{\log\left(4 \log\left(\frac{7}{\epsilon}\right)\right)} \cdot \frac{1}{2} \log\left(4 \log\left(\frac{7}{\epsilon}\right)\right) = 4 \log\left(\frac{7}{\epsilon}\right),
\end{equation*}
as wanted. To establish inequality~\eqref{EqnL2}, we need to show that
\begin{equation*}
\frac{L}{2d^2} \log\left(\frac{L}{2d^2}\right) \ge \frac{1}{d^2} \log\left(\frac{7}{\epsilon}\right).
\end{equation*}
By a similar argument as before, this holds provided
\begin{equation*}
\frac{L}{2d^2} \ge \frac{\frac{2}{d^2} \log\left(\frac{7}{\epsilon}\right)}{\log\left(\frac{1}{d^2} \log\left(\frac{7}{\epsilon}\right)\right)}.
\end{equation*}
The desired result then follows.

\section{Proofs of Propositions in Section~\ref{SecAsymmetry}}

In this appendix, we prove the propositions stated in Section~\ref{SecAsymmetry}, concerning the behavior of the diffusion estimators on asymmetric graphs. Several supporting results are stated and proved in Appendix~\ref{AppSupporting}.


\subsection{Proof of Proposition~\ref{prop: phi psi bad}}
\label{AppPropBad}

Let \(v = (j_{1}, \ldots, j_{\ell(v)})\) under the alternative labeling scheme.
From the argument in the proof of Theorem~\ref{thm: inv rumor centrality distance}, we know that
\[
\prob\left\{\phi_{T_{n}}(v^{*}) < \phi_{T_{n}}(1) \right\}
=
\prob\left\{\prod_{i = 1}^{\ell(v)}\left(1 - \frac{1}{n}|(T_{n}, 1)_{(j_{1}, \ldots, j_{i}) \downarrow}|\right) < \prod_{i = 1}^{\ell(v)}\frac{1}{n}|(T_{n}, 1)_{(j_{1}, \ldots, j_{i}) \downarrow}| \right\}.
\]
If we take the limit superior as \(n\) goes to infinity on each side,
each of the \(\frac{|(T_{n}, 1)_{(j_{1}, \ldots, j_{i})}|}{n}\) terms tends to \(1\) almost surely, by Lemma~\ref{lemma: subtree limit sizes}.
Thus,
\[
\limsup_{n \to \infty}
\prob\left\{\phi_{T_{n}}(v^{*}) < \phi_{T_{n}}(1) \right\}
=
\prob\left\{\prod_{i = 1}^{\ell(v)}(1 - 1) < \prod_{i = 1}^{\ell(v)}1\right\}
= 1,
\]
establishing the first assertion.

To prove the inequality for $\psi$, consider
\(\frac{\psi_{T_{n}}(1)}{n}\) and \(\frac{\psi_{T_{n}}(v^{*})}{n}\).
Again using Lemma~\ref{lemma: subtree limit sizes}, we have
\[
1
=
\lim_{n \to \infty}
\frac{1}{n} |(T_{n}, 1)_{v^{*} \downarrow}|
\leq 
\lim_{n \to \infty}
\frac{1}{n} \max_{v \in V(T_{n}) \setminus \{1\}} |(T_{n}, 1)_{v \downarrow}| = \frac{\psi_{T_n}(1)}{n}.
\]
On the other hand, note that
\[
\frac{\psi_{T_n}(v^*)}{n} = \lim_{n \to \infty} \frac{1}{n} \max_{v \in V(T_{n}) \setminus \{v^{*}\}} |(T_{n}, v^{*})_{v \downarrow}|
= 
\max \{B_{1}, \ldots, B_{D - 1}\},
\]
where 
\((B_{1}, \ldots, B_{D - 1})\) has a
\(\dirichlet\left(\frac{1}{D - 2}, \ldots, \frac{1}{D - 2}\right)\) 
distribution.
Since the maximum of the \(B_{i}\)'s is almost surely less than \(1\),
we see that
\[
\limsup_{n \to \infty}
\prob\left\{\frac{\psi_{T_{n}}(v^{*})}{n} < \frac{\psi_{T_{n}}(1)}{n}\right\}
= 1,
\]
which completes the proof.


\subsection{Proof of Proposition~\ref{prop: irregular confidence}}
\label{AppPropIrreg}

Note that when we examine the diffusion separately on either \(\GG_{d}\) or \(\GG_{D}\), it behaves exactly as it would on a \(d\)-regular or \(D\)-regular tree.
Additionally, we do not have to worry that the diffusion will be finite on either subtree, by Lemma \ref{lem: spread on both}.
Hence,
\begin{equation}
\label{EqnSubD}
\limsup_{n \to \infty} \prob\left\{1 \not \in H_{\phi, L}(T_{n, d}^{\circ}) | \; 1 \in T_{n, d}\right\} 
\leq 
7 \exp\left(-\frac{L}{2} \log \left(\min\left\{\frac{L(d - 2)}{4 e d^{2} \log(L)}, \frac{L}{2d^2}\right\}\right)\right),
\end{equation}
as in Corollary~\ref{CorPhiEst}, and the analogous statement holds for a diffusion starting in \(T_{n, D}\).
Using basic conditional probability, we then obtain
\begin{align*}
  &\begin{aligned}
\limsup_{n \to \infty} \prob\left\{1 \not \in H_{\phi, L}\right\}
&= 
\limsup_{n \to \infty} \prob\left\{1 \not \in H_{\phi, L}(T_{n, d}^{\circ}) | \; 1 \in T_{n, d}\right\} \prob\left\{1 \in T_{n, d}\right\} 
\\& \qquad 
+ \limsup_{n \to \infty} \prob\left\{1 \not \in H_{\phi, L}(T_{n, D}^{\circ}) | \; 1 \in T_{n, D}\right\} \prob\left\{1 \in T_{n, D}\right\} \\
& \le
7 \exp\left(-\frac{L}{2} \log \left(\min\left\{\frac{L(D - 2)}{4 e D^{2} \log(L)}, \frac{L}{2D^2}\right\}\right)\right),
  \end{aligned}
\end{align*}
where the inequality comes from the fact that 
\(\prob\left\{1 \in T_{n, d}\right\} + \prob\left\{1 \in T_{n, D}\right\} = 1\) and the bound~\eqref{EqnSubD} is larger for $D$ than for $d$. This completes the proof.


\subsection{Supporting Lemmas}
\label{AppSupporting}

This subsection contains additional results employed in the proofs derived earlier in this appendix.

\begin{lem}
Let \(\{T_{n}\}\) be a diffusion on \(\GG^{*}\), and suppose the diffusion reaches vertex \(u\) at time \(n\).
Let \(v\) be a neighbor of \(u\) such that \(v \not \in V(T_{n}).\)
Then the diffusion reaches vertex \(v\) almost surely.
\label{lemma: diffusion spread}
\end{lem}

\begin{proof}
Let \(E_{k}\) denote the number of edges between \(V(T_{k})\) and \(V(\GG^{*} \setminus T_{k})\), and let \(A\) denote the event that \(v\) is not in \(V(T_{k})\), for any \(k\).
Then we have
\begin{align*}
  &\begin{aligned}
\prob(A) 
&= 
\prod_{k = n + 1}^{\infty} \frac{E_{k} - 1}{E_{k}} \\ 
&\leq 
\prod_{k = n + 1}^{\infty} \left(1 - \frac{1}{(D+1)k}\right) \\ 
&= \exp\left(\sum_{k=n+1}^\infty \log\left(1 - \frac{1}{(D+1)k}\right)\right) \\
& \le \exp\left(\sum_{k=n+1}^\infty - \frac{1}{(D+1)k}\right) \\
& = 0.
  \end{aligned}
\end{align*}
Note that the first inequality comes from the fact that the degree of every vertex is at most \(D + 1\), so $E_k \le (D+1)k$. This proves the lemma.
\end{proof}

\begin{lem}
Let \(\{T_{n}\}\) be a diffusion on \(\GG^{*}\), and let \(v\) be any vertex.
Then the diffusion reaches \(v\) almost surely.
\label{lemma: diffusion spread everywhere}
\end{lem}

\begin{proof}
If \(v\) is the source, we are done. 
So suppose that \(v\) is not the source, and let $(v_{0}, v_{1}, \ldots, v_{k - 1}, v_{k})$ denote a path from the source to \(v\).
Let \(R_{i}\) be the event that the diffusion reaches \(v_{i}\).
Then we have
\[
\prob\left(R_{k}\right)
= 
\prob\left(R_{k} | R_{k - 1}\right) \prob\left(R_{k - 1}\right)
= 
\prod_{i = 1}^{k} \prob\left(R_{i} | R_{i - 1}\right).
\]

By Lemma \ref{lemma: diffusion spread},
each of the terms in the product is \(1\).
Thus, the diffusion reaches \(v\) almost surely.
\end{proof}

\begin{lem}
Let \(\{T_{n}\}\) be a diffusion on \(\GG^{*}\). 
Let \(u\) be a vertex in \(\GG_{d}\), and suppose the diffusion has reached \(u\).
If \(v\) lies on the path between \(u\) and \(v^{*}\), then \(\lim_{n \to \infty} 
\frac{|(T_{n}, u)_{v \downarrow}|}{n}
= 
1.
\)
If \(v\) is in \(\GG_{d}\) but does not lie on the path between \(u\) and \(v^{*}\),
then 
\( 
\lim_{n \to \infty} 
\frac{|(T_{n}, u)_{v \downarrow}|}{n}
= 
0.
\)
\label{lemma: subtree limit sizes}
\end{lem}

\begin{proof}
By hypothesis, there exists \(j\) such that \(u\) is in \(T_{j}\).
By Lemma \ref{lemma: diffusion spread everywhere}, there exists \(k\) such that 
\(v^{*}\) is in \(T_{k}\).
Let \(N = \max(j, k)\).

At time \(N\), define \(E_{n}^{d}\) to be the number of edges between \(T_{n}\) and \(\GG_{d} \setminus T_{n}\), 
and define \(E_{n}^{D}\) to be the number of edges between \(T_{n}\) and \(\GG_{D} \setminus T_{n}\).
Starting with \(n = N\), the pair
\((E_{n}^{d}, E_{n}^{D})\)
evolves according to a P\'{o}lya urn.
However, the replacement matrix in this case is diagonal with entries \(d - 2\) and \(D - 2\), respectively.
By Corollary 1 of Athreya~\cite{athreya1969},
the random variables
\(\frac{E_{n}^{D}}{E_{n}^{D} + E_{n}^{d}}\)
and 
\(\frac{E_{n}^{d}}{E_{n}^{D} + E_{n}^{d}}\)
converge to \(1\) and \(0\), respectively.
Furthermore, note that
\[
E_{n}^{D}
=
(D - 2) |(T_{n}, u)_{v^{*} \downarrow}| + 2,
\]
and 
\[
E_{n}^{d}
=
(d - 2)(n - |(T_{n}, u)_{v^{*} \downarrow}|) + 2.
\]
In particular, $\frac{E_n^D - E_n^d}{E_n^D + E_n^d} \rightarrow 1$ implies that
\begin{equation*}
\frac{(D+d-4)\frac{|(T_n, u)_{v^* \downarrow}|}{n} - (d-2)}{(D-d) \frac{|(T_n, u)_{v^* \downarrow}|}{n} + (d-2) + \frac{4}{n}} \longrightarrow 1,
\end{equation*}
from which we may conclude that $\lim_{n \rightarrow \infty} \frac{|(T_n, u)_{v^* \downarrow}|}{n} = 1$.
For \(v\) on the path between \(u\) and \(v^{*}\), we see that
\(|(T_{n}, u)_{v^{*} \downarrow}| \leq |(T_{n}, u)_{v \downarrow}|,\)
since the subtree rooted at \(v\) contains the subtree rooted at \(v^{*}\).
If \(v\) is not on the path between \(u\) and \(v^{*}\), we have
\[
\lim_{n \to \infty}
\frac{|(T_{n}, u)_{v \downarrow}|}{n}
\leq 
\lim_{n \to \infty} 
\frac{n - |(T_{n}, u)_{v^{*} \downarrow}| }{n}
= 0,
\]
which completes the proof.
\end{proof}


\begin{lem}
Let \(\{T_{n}\}\) be a diffusion on \(\GG^{*}\).
Let \(T_{n, d}\) and \(T_{n, D}\) be the diffusion subtrees on \(\GG_{d}\) and \(\GG_{D}\), respectively.
Then
\[
\lim_{n \to \infty} |T_{n, d}| 
= 
\lim_{n \to \infty} |T_{n, D}|
= 
\infty,
\]
almost surely.
\label{lem: spread on both}
\end{lem}

\begin{proof}
We  prove the claim for \(|T_{n, d}|\), since the proof for \(|T_{n, D}|\) is  identical.
Let \(M\) be a positive integer.
Pick a set of vertices \(S_{M}^{d} = \{v_{1}, \ldots, v_{M}\}\)  in \(\GG_{d}\)  of size \(M\).
If we let \(A_{i}\) be the event that the diffusion fails to reach \(v_{i}\), we have
\[
\prob\left(\bigcup_{i = 1}^{M} A_{i}\right)
\leq 
\sum_{i = 1}^{M} \prob\left(A_{i}\right)
= 0,
\]
by Lemma \ref{lemma: diffusion spread everywhere}.
Thus, there exists an \(N\) such that \(|T_{N, d}| \geq M\), almost surely.
\end{proof}


\section{Auxiliary Lemmas}
\label{AppAuxiliary}

This appendix contains further technical results used in the proofs of the paper.

The following lemma is a stochastic domination result for beta random variables:
\begin{lem}
Let \(B\) and \(B'\) be  \(\betavar\left(\frac{1}{d - 2}, \frac{d - 1}{d - 2}\right)\) and \(\betavar\left(\frac{1}{d - 2}, 1\right)\) random variables, respectively.
Then \(B'\) stochastically dominates \(B\); i.e.,
\[
\prob\left\{B \geq t\right\}
\leq 
\prob\left\{B' \geq t\right\},
\]
for any \(t\).
\label{lemma: beta dominance}
\end{lem}

\begin{proof}
Consider two P\'{o}lya urns.
Let the contents of the first and second urns be \((C_{n}^{(d - 1)}, D_{n}^{(d - 1)})\) and \((C_{n}^{(d - 2)}, D_{n}^{(d - 2)})\), respectively, after the \(n^\text{th}\) draw.
The initial conditions are \((1, d - 1)\) and \((1, d - 2)\), respectively,
and the replacement matrix is \((d - 2)I_{2}\) for both urns.

We only need to show that
\begin{equation}
\prob\left\{\frac{1}{d + n(d - 2)}C_{n}^{(d - 1)} \geq t\right\}
\leq
\prob\left\{\frac{1}{d - 1 + n(d - 2)} C_{n}^{(d - 2)} \geq t\right\},
\label{eqn: stoch dom condition}
\end{equation}
since taking the limit as \(n \to \infty\) would give
\[
\prob\left\{B \geq t\right\}
\leq 
\prob\left\{B' \geq t\right\}.
\]

We can prove equation~\eqref{eqn: stoch dom condition} by a coupling argument over a suitable probability space.
Let \(U_{1}, \ldots, U_{n}\) be i.i.d.\ \(\text{Uniform}[0, 1]\) random variables.
Define \(F_{0}^{(d - 1)} = F_{0}^{(d -2)} = 1\), and \(F_{i}^{(d - 1)}\) and \(F_{i}^{(d - 2)}\), for \(i > 0\), to be
\[
F_{i}^{(j)}
= 
\begin{cases}
F_{i - 1}^{(j)} + d - 2, 
& \text{if } U_{i} \leq \frac{F_{i - 1}^{(j)}}{1 + j + (i - 1)(d - 2)}, \\ 
F_{i - 1}^{(j)},
& \text{if } U_{i} >\frac{F_{i - 1}^{(j)}}{1 + j + (i - 1)(d - 2)}.
\end{cases}
\]
By design, we have \(\prob\left\{F_{i}^{(j)} \geq t\right\} = \prob\left\{B_{i}^{(j)} \geq t\right\}\)
for both \(j = d - 1\) and \(j = d - 2\), so it suffices to show that \(F_{n}^{(d - 1)} \leq F_{n}^{(d - 2)}\), for all \(n\).
This is simple to show by induction.
The inequality is true for \(n = 0\).
Now suppose \(F_{i}^{(d - 1)} \leq F_{i}^{(d - 2)}\) for all \(i < I\), where $I \ge 1$.
Then \(U_{I}\) is 
in either
\(\left[0, \;  \frac{F_{I - 1}^{(d - 1)}}{d + (I - 1)(d - 2)}\right],\)
in 
\(\left(\frac{F_{I - 1}^{(d - 1)}}{d + (I - 1)(d - 2)}, \; \frac{F_{I - 1}^{(d - 2)}}{d - 1 + (I - 1)(d - 2)}\right] ,\)
or in
\(\left(\frac{F_{I - 1}^{(d - 2)}}{d - 1 + (I - 1)(d - 2)}, \;  1\right].\)
In the first case, we have
\[
F_{I}^{(d - 1)}
= 
F_{I - 1}^{(d-1)} + d - 2
\leq 
F_{I - 1}^{(d - 2)} + d - 2
= 
F_{I}^{(d - 2)}.
\]
In the second case, we have
\[
F_{I}^{(d - 1)}
= 
F_{I - 1}^{(d-1)}
<
F_{I - 1}^{(d - 2)} + d - 2
= 
F_{I}^{(d - 2)}.
\]
In the third case, we have
\[
F_{I}^{(d - 1)}
= 
F_{I - 1}^{(d-1)}
\leq 
F_{I - 1}^{(d - 2)}
= 
F_{I}^{(d - 2)}.
\]
This proves that \(F_{n}^{(d - 1)} \leq F_{n}^{(d - 2)}\) for all \(n\),
which completes the proof of the lemma.
\end{proof}

We also have the following lemma concerning concentration of a function of beta random variables:

\begin{lem}
Let 
\(B = \prod_{i = 1}^{\infty}  \min\left(-\sum_{k = 1}^{i} \log(B_{k}), \; 1\right)\),
where the 
\(B_{k}\)'s
are i.i.d.\ \(\betavar\left(\frac{1}{d - 2}, \; 1\right)\)
random variables.
Then
\[
\prob\left\{
B \leq s \right\}
\leq 
6 s^{\frac{1}{4}},
\]
for all $s > 0$.
\label{lemma: beta bound}
\end{lem}

\begin{proof}
The proof is essentially the same as the proof of Lemma 2 in Bubeck et al.~\cite{bubeck},
but we include it here for completeness.
For brevity, let \(P\) denote the probability we wish to bound.
First note that \(-\log(B_{k})\) is an exponential random variable with parameter \(\frac{1}{d - 2}\).
This can be seen easily by computing
\begin{align*}
  &\begin{aligned}
\prob\left\{-\log(B_{k}) \leq s\right\} 
&= 
\prob\left\{B_{k} \geq e^{-s}\right\} \\ 
&= 
\frac{1}{d - 2} \int_{e^{-s}}^{1} x^{\frac{1}{d - 2} - 1} \; dx \\ 
&= 
1 - e^{-\frac{s}{d - 2}}.
  \end{aligned}
\end{align*}
Also observe that \(-\sum_{k = 1}^{\infty} \log(B_{k}) > 1\), almost surely.
Thus, for some \(j\), we see that we have
\(-\sum_{k = 1}^{j} \log(B_{k}) \leq 1\) and
\(-\sum_{k = 1}^{j + 1} \log(B_{k}) > 1\).
Hence,
\begin{align}
& P
= 
\prob\left\{\exists j: -\sum_{k = 1}^{j }\log(B_{k}) \leq 1,
\text{ and }  -\sum_{k = 1}^{j + 1} \log(B_{k}) > 1,
 \text{ and }  \prod_{i = 1}^{j}\left(-\sum_{k = 1}^{i} \log(B_{k})\right) \leq s
\right\} \notag  \\ 
& \; \leq 
\sum_{j = 1}^{\infty} \min\left(
\prob\left\{-\sum_{k = 1}^{j}\log(B_{k}) \leq 1 \right\}, 
\prob\left\{-\sum_{k = 1}^{j + 1} \log(B_{k}) > 1
 \text{ and }  \prod_{i = 1}^{j}\left(-\sum_{k = 1}^{i} \log(B_{k})\right) \leq s\right\}
\right).
\label{eqn: min probs}
\end{align}
We now analyze the probability of each event appearing in inequality~\eqref{eqn: min probs}.
It is easy to see  that 
\(-\sum_{k = 1}^{j }\log(B_{k})\) has a \(\text{Gamma}\left(j, \frac{1}{d - 2}\right)\) distribution,
by considering characteristic functions.
Thus,
\begin{align}
\label{EqnB1}
  &\begin{aligned}
\prob\left\{-\sum_{k = 1}^{j}\log(B_{k}) \leq 1, \right\} 
&= 
\frac{(d - 2)^{-j}}{j!} \int_{0}^{1} x^{j - 1} e^{-\frac{x}{d - 2}} \; dx \\ 
&\leq 
\frac{(d - 2)^{-j}}{(j - 1)!} \int_{0}^{1} x^{j - 1} \; dx \\ 
&=
\frac{(d - 2)^{-j}}{j!}.
  \end{aligned}
\end{align}
To analyze the second quantity in inequality~\eqref{eqn: min probs}, we write
\[
\prod_{i = 1}^{j}\left(-\sum_{k = 1}^{i} \log(B_{k})\right)
= 
-\sum_{\ell = 1}^{j + 1} \log(B_{k})
\prod_{i = 1}^{j} \left(\frac{-\sum_{k = 1}^{i} \log(B_{k})}{-\sum_{\ell = 1}^{j + 1} \log(B_{k})} \right).
\]
By properties of exponential random variables, the vector 
\(\left(-\sum_{k = 1}^{i} \log(B_{k}) / (-\sum_{\ell = 1}^{j + 1} \log(B_{\ell}) )\right)_{i=1, \dots, j}\)
is equal in distribution to the vector \((U_{(i)})_{i = 1, \ldots, j}\)  where 
\(U_{1}, \ldots, U_{j}\) are i.i.d.\ 
\(\text{Uniform}[0, 1]\) random variables.
Thus, we have
\begin{align}
\label{EqnB2}
  &\begin{aligned}
 \prob\left\{-\sum_{k = 1}^{j + 1} \log(B_{k}) > 1
 \text{ and }  \prod_{i = 1}^{j}\left(-\sum_{k = 1}^{i} \log(B_{k})\right) \leq s\right\}
&\leq 
\prob\left\{\prod_{i = 1}^{j} U_{i} \leq s\right\} \\ 
&\leq 
\expect\left[\frac{\sqrt{s}}{\prod_{i = 1}^{j} \sqrt{U_{i}}}\right] \\ 
& \leq 
2^{j} \sqrt{s},
  \end{aligned}
\end{align}
where the second inequality follows from Markov's inequality. Applying the fact \(\min(a, b) \leq \sqrt{ab}\) to 
inequality~\eqref{eqn: min probs} and using the bounds~\eqref{EqnB1} and~\eqref{EqnB2}, we then obtain
\begin{align*}
  &\begin{aligned}
P 
&\leq 
\sum_{j = 1}^{\infty} \sqrt{\left(\frac{2}{d - 2}\right)^{j} \frac{\sqrt{s}}{j!}} \\ 
&\leq 
6 s^{\frac{1}{4}},
  \end{aligned}
\end{align*}
which completes the proof of the lemma.
\end{proof}

\end{appendices}

\end{document}